\documentclass{amsart}

\usepackage{latexsym,enumerate}
\usepackage{amsmath,amsthm,amsopn,amstext,amscd,amsfonts,amssymb}
\usepackage[ansinew]{inputenc}
\usepackage{verbatim}
\usepackage{graphicx}
\usepackage{pstricks}
\usepackage{wasysym}
\usepackage[all]{xy}
\usepackage{epsfig} 
\usepackage{graphicx,psfrag} 
\usepackage{subfigure}
\usepackage{pstricks,pst-node}

\usepackage{multicol}
\usepackage{color}
\usepackage{colortbl}

\usepackage{showlabels}

\newtheorem{thm}{Theorem}[section]
\newtheorem{cor}[thm]{Corollary}
\newtheorem{lem}[thm]{Lemma}
\newtheorem{prop}[thm]{Proposition}
\newtheorem{defn}[thm]{Definition}
\newtheorem{rem}[thm]{Remark}

\DeclareMathOperator{\diam}{diam}

\DeclareMathOperator{\diameff}{effdiam}

\newcommand{\spb}[1]{\smallskip}
\newcommand{\mpb}[1]{\medskip}
\newcommand{\bpb}[1]{\bigskip}

\renewcommand{\d}{\delta}

\newcommand{\g}{\gamma}

\begin{document}
\DeclareGraphicsExtensions{.jpg,.pdf,.mps,.png}
\title{Bounds on Gromov Hyperbolicity Constant} 
\author[Ver\'onica Hern\'andez]{Ver\'onica Hern\'andez$^{(1)}$}
\address{Departamento de Matem\'aticas, Universidad Carlos III de Madrid,
Avenida de la Universidad 30, 28911 Legan\'es, Madrid, Spain}
\email{vehernan@math.uc3m.es}

\author[Domingo Pestana]{Domingo Pestana}
\address{Departamento de Matem\'aticas, Universidad Carlos III de Madrid,
Avenida de la Universidad 30, 28911 Legan\'es, Madrid, Spain}
\email{dompes\@@math.uc3m.es}

\author[Jos\'e M. Rodr{\'\i}guez]{Jos\'e M. Rodr{\'\i}guez}
\address{Departamento de Matem\'aticas, Universidad Carlos III de Madrid,
Avenida de la Universidad 30, 28911 Legan\'es, Madrid, Spain}
\email{jomaro@math.uc3m.es}
\thanks{$^{(1)}$ Supported in part by a
grant from Ministerio de Econom{\'\i}a y Competititvidad (MTM 2013-46374-P), Spain.}

\date{}

\maketitle{}


\begin{abstract}
 If $X$ is a geodesic metric space and $x_{1},x_{2},x_{3} \in X$, a  geodesic triangle
 $T=\{x_{1},x_{2},x_{3}\}$ is the union of the three  geodesics $[x_{1}x_{2}]$, $[x_{2}x_{3}]$ and $[x_{3}x_{1}]$ in $X$. The space $X$ is $\delta$-hyperbolic in the Gromov sense if any side
 of $T$ is contained in a $\delta$-neighborhood of the union of the two
 other sides, for every geodesic triangle $T$ in $X$.
 If $X$ is hyperbolic, we denote by
 $\delta(X)$ the sharp hyperbolicity constant of $X$, i.e. $\delta(X) =\inf \{ \delta\geq 0:\hspace{0.3cm} X \hspace{0.2cm} \text{is} \hspace{0.2cm} \delta \text{-hyperbolic} \}.$  To compute the hyperbolicity constant is a very hard problem. Then it is natural to try to bound the hyperbolycity constant in terms of some parameters of the graph. Denote by
 $\mathcal{G}(n,m)$ the set of graphs $G$ with $n$  vertices and $m$ edges, and such that every  edge has length $1$. In this work we estimate  $A(n,m):=\min\{\delta(G)\mid G \in \mathcal{G}(n,m) \}$ and  $B(n,m):=\max\{\delta(G)\mid G \in \mathcal{G}(n,m) \}$. In particular, we obtain good bounds for $B(n,m)$, and 
 we compute the precise value of $A(n,m)$ for all values of $n$ and $m$. Besides, we apply these results to random graphs.

\end{abstract}



\section{Introduction}



Gromov hyperbolicity was introduced by the Russian mathematician Mikhail Leonidovich Gromov  in the setting of geometric group theory
\cite{G1}, \cite{G3}, \cite{GH}, \cite{CDP}, but has played an increasing role in analysis on general metric
spaces \cite{BHK}, \cite{BS}, \cite{BB}, with applications to the Martin boundary, invariant metrics in several
complex variables \cite{BBonk} and extendability of Lipschitz mappings \cite{La}.\smallskip

The theory of Gromov hyperbolic spaces was used initially for the study of finitely generated groups, where it was demonstrated to have an enormous practical importance. This theory was applied principally to the study of automatic groups (see \cite{O}), which plays an important role in sciences of the computation. The concept of hyperbolicity appears also in discrete mathematics, algorithms
and networking. Another important application of these spaces is the secure transmission of information by internet. In particular, the hyperbolicity plays an important role in the spread of viruses through the network (see \cite {K21,K22}).
The hyperbolicity is also useful in the study of DNA data (see \cite{BHB1}).\smallskip

The study of mathematical properties of Gromov hyperbolic spaces and its applications is a topic of recent and increasing interest in graph theory; see, for instance, \cite{BRST,BHB1,CPeRS,CRSV,CDEHV,K50,
K21,K22,K24,K56,MRSV,MRSV2,PRT1,RSVV, T}.\smallskip

Last years several researchers have been interested in showing that metrics used in geometric function theory are Gromov hyperbolic. For instance, the Gehring-Osgood $j$-metric
is Gromov hyperbolic; and the Vuorinen $j$-metric is not Gromov hyperbolic except in the punctured space (see \cite{Ha}). The study of Gromov hyperbolicity of the quasihyperbolic and the Poincar\'e metrics is the subject of
\cite{
BB,BHK,
HPRT,
PRT1,PT,
T}. In particular, the equivalence of the hyperbolicity of Riemannian manifolds and the hyperbolicity
of a simple graph was proved in \cite{PRT1,T}, hence, it is useful to know hyperbolicity criteria for graphs.\smallskip

Now, let us introduce the concept of Gromov hyperbolicity and the main results concerning this theory. For detailed expositions about Gromov hyperbolicity, see e.g. \cite{ABCD}, \cite{GH}, \cite{CDP}  or \cite{Va}.\smallskip


If $X$ is a metric space we say that the curve $\g:[a,b]\longrightarrow 
X$ is a
\emph{geodesic} if we have $L(\g|_{[t,s]})=d(\g(t),\g(s))=|t-s|$ for 
every $s,t\in [a,b]$
(then $\gamma$ is equipped with an arc-length parametrization).
The metric space $X$ is said \emph{geodesic} if for every couple of 
points in
$X$ there exists a geodesic joining them; we denote by $[xy]$
any geodesic joining $x$ and $y$; this notation is ambiguous, since in 
general we do not have uniqueness of
geodesics, but it is very convenient.
Consequently, any geodesic metric space is connected.
If the metric space $X$ is
a graph, then the edge joining the vertices $u$ and $v$ will be denoted 
by $[u,v]$.\smallskip

In order to consider a graph $G$ as a geodesic metric space, identify 
(by an isometry)
any edge $[u,v]\in E(G)$ with the interval $[0,1]$ in the real line;
then the edge $[u,v]$ (considered as a graph with just one edge)
is isometric to the interval $[0,1]$.
Thus, the points in $G$ are the vertices and, also, the points in the 
interior
of any edge of $G$.
In this way, any graph $G$ has a natural distance defined on its points, 
induced by taking the shortest paths in $G$,
and we can see $G$ as a metric graph.
Throughout this paper, $G=(V,E)$ denotes a simple connected graph such 
that every edge has length $1$.
These properties guarantee that any graph is a geodesic metric space.
Note that to exclude multiple edges and loops is not an important loss 
of generality, since \cite[Theorems 8 and 10]{BRSV2} reduce the problem 
of computing
the hyperbolicity constant of graphs with multiple edges and/or loops to 
the study of simple graphs.\smallskip

If $X$ is a geodesic metric space and $J=\{J_1,J_2,\dots,J_n\}$
is a polygon with sides $J_j\subseteq X$, we say that $J$ is $\d$-{\it thin} if for
every $x\in J_i$ we have that $d(x,\cup_{j\neq i}J_{j})\le \d$.
In other words, a polygon is $\d$-thin if each of its sides is contained in the $\d$-neighborhood
of the union of the other sides.
We denote by $\d(J)$ the sharp thin constant of $J$, i.e.,
$\d(J):=\inf\{\d\ge 0| \, J \, \text{ is $\d$-thin}\,\}\,. $
If $x_1,x_2,x_3\in X$, a {\it geodesic triangle} $T=\{x_1,x_2,x_3\}$ is
the union of the three geodesics $[x_1x_2]$, $[x_2x_3]$ and
$[x_3x_1]$. The space $X$ is $\d$-\emph{hyperbolic} $($or satisfies
the {\it Rips condition} with constant $\d)$ if every geodesic
triangle in $X$ is $\d$-thin. We denote by $\d(X)$ the sharp
hyperbolicity constant of $X$, i.e., $\d(X):=\sup\{\d(T)| \, T \,
\text{ is a geodesic triangle in }\,X\,\}.$ We say that $X$ is
\emph{hyperbolic} if $X$ is $\d$-hyperbolic for some $\d \ge 0$. If
$X$ is hyperbolic, then $ \d(X)=\inf\{\d\ge 0| \, X \, \text{ is
$\d$-hyperbolic}\,\}.$\smallskip

The following are interesting examples of hyperbolic spaces.
Every bounded metric space $X$ is $(\diam X)$-hyperbolic.
The real line $\mathbb{R}$ is $0$-hyperbolic due to any point of a geodesic triangle in the real line belongs to two sides of the triangle simultaneously.
The Euclidean plane $\mathbb{R}^2$ is not hyperbolic,  since the midpoint of a side on a large equilateral triangle is far from all points in the other two sides.
A normed vector space $E$ is hyperbolic if and only if $\dim\ E=1$. 
Every simply connected complete Riemannian manifold with sectional curvature verifying $K\leq -k^2$, for some positive constant $k$, is hyperbolic (see, e.g., \cite[p.52]{GH}).
The graph $\Gamma$ of the routing infraestructure of the Internet is also empirically  shown to be hyperbolic (see \cite{Bar}).\smallskip

The main examples of hyperbolic graphs are trees.
In fact, the hyperbolicity constant of a geodesic metric space can be viewed as a measure of
how ``tree-like'' the space is, since those spaces $X$ with $\delta(X) = 0$ are precisely the metric trees.
This is an interesting subject since, in many applications, one finds that the borderline between tractable and intractable cases may be the tree-like degree of the structure to be dealt with
(see, e.g., \cite{CYY}).\smallskip

For a general graph deciding whether or not a space is hyperbolic seems an unabordable problem.
Therefore, it is interesting to study the hyperbolicity of particular classes of graphs.
The papers \cite{BRST,BHB1,CCCR,CRS2,CRSV,MRSV2,PeRSV,PRSV,R,Si,WZ}  study the hyperbolicity of, respectively, complement of graphs, chordal 
graphs, strong product graphs, corona and join of graphs, line graphs, Cartesian product graphs, cubic graphs, tessellation 
graphs, short graphs, median graphs and  $k$-chordal graphs. In \cite{CCCR,CRS2,MRSV2} the authors characterize the hyperbolic 
product graphs (for strong product, corona and join of graphs, and Cartesian product) in terms of properties of the factor graphs. In this work we study the hyperbolicity constant of the graphs with $n$  vertices and $m$ edges.\smallskip

Let $\mathcal{G}(n,m)$ be the set of graphs $G$ with $n$  vertices and $m$ edges, and such that every  edge has length $1$. If $m=n-1$, then every $ G \in \mathcal{G}(n,m) $ is a tree and $\delta(G) =0$. On the other hand, the complete graph $K_{n}$ belongs to $ \mathcal{G}(n,m)$ with $m=\displaystyle{n\choose 2}$. Thus we consider $n-1 \leq m\leq \displaystyle{n\choose 2}$.\smallskip

Let us define 
$$A(n,m):=\min\{\delta(G)\mid G \in \mathcal{G}(n,m) \},$$  
$$B(n,m):=\max\{\delta(G)\mid G \in \mathcal{G}(n,m) \}.$$ 

Our aim in this paper is to estimate  $A(n,m)$ and  $B(n,m)$. In particular, we obtain good bounds for $B(n,m)$, and  we compute the precise value of $A(n,m)$ for all values of $n$ and $m$.\smallskip

The structure of this paper is as follows. In the next section we consider some previous results regarding hyperbolicity.
In Section $3$ we prove an upper bound for $B(n,m)$ (see Theorem \ref{t:bound2}).  Also, we find a lower bound for $B(n,m)$ in Section $4$ (see Theorem \ref{lowerbound}). In Section 5 we give an estimation of the differece between the upper and the lower bounds of $B(n,m)$.  One of the main results of this work is Theorem  \ref{bound-final}, which gives the precise value of $A(n,m)$.
We conclude this paper with Section $7$, where we discuss the applications of our previous results to random graphs.\smallskip

\section{Upper Bound of $B(n,m)$}

First, our purpose is to find an upper bound for $B(n,m)$. In order to simplify this proof, we prove some technical lemmas. We begin by proving Lemma \ref{l:lemma 1}. In order to prove it, we will use Karush-Kuhn-Tucker necessary conditions for nonlinear optimization problems with inequality constraints.\smallskip

Let $X$ be a non-empty open set of $\mathbb{R}^n$ and $f, g_{j}\hspace{0.2cm}(j=1,\dots,k) $ functions of $X\subseteq \mathbb{R}^n$ in  $\mathbb{R}$. Consider the problem:
$$\quad \text{P:} \quad \min_{x\in V}  f(x), $$
with $ \quad V:=\{x\in X \mid g_{j}(x)\leq0 \hspace{0.2cm}(j=1,\dots,k)\}. $\smallskip

 Given $x^{*} \in V$, let $I(x^{*})$ be the set of subscripts $j$ for which $g_{j}(x^{*})=0$.
\begin{defn}
We say that a point  $x^{*}\in V$ is  \emph{regular}   if the vectors $\nabla g_{j}(x^{*}) \hspace{0.2cm}  (j\in I(x^{*}))$ are linearly independent.
\end{defn}
\begin{thm}
\label{kkt}
Let $x^{*}$ be a point in $V$.
Suppose that $f, g_{j} \hspace{0.2cm} (j\in I(x^{*}))  $ are continuously differentiable functions and $g_{j} \hspace{0.2cm} (j\notin I(x^{*})) $ are continuous functions at $x^{*}$. If $x^{*}$ is a regular point and a local minimum of $f$ in $V$, then there exist unique scalars $\mu_{j} \hspace{0.2cm} (j\in I(x^{*})) $ such that: 
$$ \nabla f(x^{*})+ \Sigma_{j\in I(x^{*})}\mu_{j}\nabla g_{j}(x^{*})=0, $$
$$ \mu_{j}\geq 0, \hspace{0.5cm} j\in I(x^{*}).$$
\end{thm}

The above conditions can be written as:
$$ \nabla f(x^{*})+ 
\Sigma_{j=1}^{k}\mu_{j}\nabla g_{j}(x^{*})=0 ,$$
$$ \mu_{j} g_{j}(x^{*})=0,\hspace{0.5cm} j=1,\dots, k, $$
$$ \mu_{j} \geq 0, \hspace{0.5cm} j=1,\dots,k. $$

Consider $G \in \mathcal{G}(n,m)$. Fix $\text{diam}\,V(G)=r$ and choose $u,v\in V(G)$ such that $d(u,v)=r$. Let $k_{j}=\#\{ w\in V(G):\hspace{0.2cm}d(w,u)=j \}$ $(0\leq j \leq r) $. The number of edges that we must eliminate from the complete graph of $n$ vertices in order to obtain $G$ is at least
$$f_{r}(k_{1},k_{2},\dots,k_{r}):=\Sigma_{t=2}^{r}k_{t}\Sigma_{s=0}^{t-2}k_{s}. $$

In the next result we compute the minimum value of $f_{r}$ such that $\text{diam}\,V(G)=r$ with $k_{j}\geq2,\hspace{0.2cm} (0\leq j \leq r-1)$.

\begin{lem} \label{l:lemma 1}
Consider the following optimization problem:
$$ \Delta_{r}:=\min_{x\in W} f_{r}, \quad \text{with} \quad  f_{r}(k_{1},k_{2},\dots,k_{r}):=\Sigma_{t=2}^{r}k_{t}\Sigma_{s=0}^{t-2}k_{s},\hspace{0.5cm} 2 \leq r \leq \dfrac{n}{2},$$
$$ \quad \text{and W:=\hspace{0.2cm}} \{ k_{0}=1, \hspace{0.5cm} k_{j}\geq2,  \quad \text{if} \quad 1\leq j\leq r-1, \quad k_{r}\geq1, $$
$$ 1+k_{1}+k_{2}+\dots+k_{r}=n \}. $$
Then $\Delta_{2}=1$,  $\Delta_{3}=n-1$ and $\Delta_{r}= 2n(r-3)-2r^{2}+6r+5 $ for $ r\geq4$.
\end{lem}
\begin{proof}
If $r=2$, then $f_{2}(k_{1},k_{2})=k_{2}$, with $k_{2}\geq1$. Hence $\Delta_{2}=1$. \smallskip

Consider now $r\geq 3$. The set $W$ can be written as:
$$ W= \{ k_{0}=1, \hspace{0.5cm} g_{j}=-k_{j}+2\leq0, \quad \text{if} \quad 1\leq j \leq r-1 ,$$
$$ g_{r}=-k_{r}+1\leq0,\quad  h=1+k_{1}+k_{2}+\dots+k_{r}-n=0\}.$$

Note that if $ W \neq \emptyset $, then $n= 1+\Sigma_{t=1}^{r}k_{t}\geq 1+2(r-1)+1 $ and $ 2r\leq n$. Conversely, if $ 2r\leq n$, then $ W \neq \emptyset $. Hence, we are assuming $ 2r\leq n$.\smallskip

We eliminate a variable of our problem by solving $ k_ {r} $ in the equality restriction. Substituting the expression obtained in $ f_{r}$, the original problem is reduced to the following:
$$ \Delta_{r}= \min_{x\in W^{1}}  f_{r}^{1}, \quad \text{with} \quad f_{r}^{1}(k_{1},k_{2},\dots,k_{r-1}) :=\Sigma_{t=2}^{r-1}k_{t}\Sigma_{s=0}^{t-2}k_{s} +$$ $$+ (n-\Sigma_{s=0}^{r-1}k_{s})\Sigma_{s=0}^{r-2}k_{s},$$
$$\quad \text{and } W^{1}:=\hspace{0.2cm} \{  k_{0}=1, \hspace{0.2cm} g_{j}=-k_{j}+2\leq0,  \quad \text{if} \quad 1\leq j \leq r-1, $$
$$ {g}_{r}=-k_{r}+1=2-n+\Sigma_{s=1}^{r-1}k_{s}\leq0\}. $$

Note that the vectors $\{\nabla g_{j}(x^{*}),\hspace{0.2cm} j=1,\dots,r \}$ are linearly dependent but become a linearly independent set by removing any of its elements. Therefore,  it suffices to consider that at least one of the coefficients $\mu_{j}$  is zero, so that the point is regular. \smallskip

Let us consider first the case in which $x^{*} $ is not a regular point  (then $ g_{j}(x^{*})=0$ for every $  1\leq j\leq r$). Hence:
 $$h=1+2(r-1)+1-n=0 \quad  \Rightarrow \quad  2r=n .$$
 
Therefore, $x^{*}=(2,\dots,2), $  $W^{1}=\{x^{*}\}$ and evaluating $f_{r}$ at $x=(x^{*},1)=(2,\dots,2,1)$ we get:
\begin{eqnarray*}
f_{r}(x)&=&\Sigma_{t=2}^{r-1}2(1+\Sigma_{s=1}^{t-2}2)+(1+\Sigma_{s=1}^{r-2}2)\\
&=& 2\Sigma_{t=2}^{r-1}(2t-3)+2r-3  \\
&=& (1+2r-5)(r-2)+2r-3  \\
&=& 2{r}^2 -6r+5,
\end{eqnarray*}
and then $\Delta_{r}= 2r^{2}-6r+5$. \smallskip

Now assume that the minimum point is regular, then $ g_{j} \neq 0$ for some $1\leq j \leq r $ and we can apply Theorem \ref{kkt}. Since: 
$$ \frac{\partial f_{r}^{1} }{\partial k_{r-1}}=\Sigma_{s=0}^{r-3}k_{s}-\Sigma_{s=0}^{r-2}k_{s}=-k_{r-2} ,$$
we conclude that  the following equality must be satisfied at a regular minimum point:
 $$ 
 \begin{pmatrix}
* \\
\vdots \\
* \\
-k_{r-2}
\end{pmatrix} +  
\mu_{1}\begin{pmatrix}
-1 \\
0 \\
\vdots\\
0
\end{pmatrix}+\dots+
\mu_{r-1}\begin{pmatrix}
0 \\
0\\
\vdots\\
-1
\end{pmatrix}+
\mu_{r}\begin{pmatrix}
1 \\
1 \\
\vdots \\
1 
\end{pmatrix}=
\begin{pmatrix}
0 \\
0 \\
\vdots \\
0 
\end{pmatrix}
$$
with $ \mu_{j}\geq 0 \hspace{0.2cm}$ for $j=1,\dots,r$.  \smallskip

Assuming that  $\mu_{r}=0$, from the previous expression we obtain that $-k_{r-2}=\mu_{r-1}$.
The restriction $g_{r-2}\leq0$ of the problem and the positivity of the coefficient $ \mu_ {r-1} $ implies that $-2\geq -k_{r-2}=\mu_{r-1}\geq0$  and this is a contradiction, therefore $\mu_{r}>0$.\smallskip

Considering the condition  $ \mu_{r} g_{r}(x^{*})=0 $ we deduce that ${g}_{r}=-k_{r}+1=0$ and $ k_{r}=1$.\smallskip

We write again the optimization problem, with $k_{r}=1$:
$$ \Delta_{r}= \min_{x\in W^{2}} f_{r}^{2}, \quad \text{with} \quad f_{r}^{2} (k_{1},k_{2},\dots,k_{r-1}):=\Sigma_{t=2}^{r-1}k_{t}\Sigma_{s=0}^{t-2}k_{s} + \Sigma_{s=0}^{r-2}k_{s}.$$
$$\quad \text{and } W^{2}:=\hspace{0.2cm} \{  k_{0}=1, \hspace{0.2cm} k_{j}\geq 2,  \quad \text{if} \quad 1\leq j \leq r-1, $$
$$ k_{1}+k_{2}+...+k_{r-1}=n-2 \}. $$
If $r=3$, then $f_{3}^{2}(k_{1},k_{2})=k_{2}+1+k_{1}$, with $k_{1},k_{2}\geq2$ and $k_{1}+k_{2}=n-2$. Hence, $\Delta_{3}=n-1$.\smallskip

Consider now $r\geq 4$. Note that:
\begin{eqnarray*}
f_{r}^{2}&=& k_{2}+ \Sigma_{t=3}^{r-1}k_{t}(1+\Sigma_{s=1}^{t-2}k_{s})+1+\Sigma_{s=1}^{r-2}k_{s}\\
&=& 1-k_{1}-k_{r-1} +2\Sigma_{t=1}^{r-1}k_{t}+\Sigma_{t=3}^{r-1}\Sigma_{s=1}^{t-2}k_{t}k_{s}\\
&=& 2n-3-k_{1}-k_{r-1} +\Sigma_{t-2\geq s}k_{t}k_{s}, \quad \text{with} \quad \Sigma_{t=1}^{r-1}k_{t}=n-2. 
\end{eqnarray*}

Consider now the expression $(\Sigma_{t=1}^{r-1} k_{t})^{2}$:
$$ (\Sigma_{t} k_{t})^{2}= \Sigma_{t} k_{t}^{2} +2\Sigma_{t-1\geq s}k_{t}k_{s}= \Sigma_{t} k_{t}^{2} +2\Sigma_{t-1=s}k_{t}k_{s}+2\Sigma_{t-2\geq s}k_{t}k_{s}.$$

Moreover, we can write:
\begin{eqnarray*}
\Sigma_{t-2\geq s}k_{t}k_{s}
&=& \frac{1}{2}(\Sigma_{t=1}^{r-1} k_{t})^{2}-\frac{1}{2}\Sigma_{t=1}^{r-1} k_{t}^{2}-\Sigma_{t=2}^{r-1} k_{t}k_{t-1}\\
&=& \frac{1}{2}(n-2)^{2}-\frac{1}{2}\Sigma_{t=1}^{r-1} k_{t}^{2}-\Sigma_{t=2}^{r-1} k_{t}k_{t-1}. 
\end{eqnarray*}

Thus we have deduced that $\Delta_{r}= \min_{x\in W^{3}} f_{r}^{3},$ with:
$$ f_{r}^{3}(k_{1},k_{2},\dots,k_{r-1}):=\frac{1}{2}n^{2}-1
-k_{1}-k_{r-1} -\frac{1}{2}\Sigma_{t=1}^{r-1} k_{t}^{2}  -\Sigma_{t=2}^{r-1}k_{t}k_{t-1},$$
and $W^{3}:= \{  k_{j}\geq 2 \quad \text{if} \quad 1\leq j \leq r-1,\hspace{0.2cm}  k_{1}+k_{2}+\dots+k_{r-1}=n-2 \}.$\smallskip

This formulation allows us to see  that the problem is symmetric in the variables $k_{t}$ and $k_{r-t}$ for every $1\leq t \leq r-1$.\smallskip

Substituting $ k_{r}=1$ and $  k_{r-1}=n-2-\Sigma_{t=1}^{r-2}k_{t}$ in $f_{r}$ we obtain $ \Delta_{r}= \min_{x\in W^{4}} f_{r}^{4}$, with:
 $$ f_{r}^{4}(k_{1},k_{2},\dots,k_{r-2}):=(n-2-\Sigma_{t=1}^{r-2}k_{t})\Sigma_{s=0}^{r-3}k_{s}+ \Sigma_{t=2}^{r-2}k_{t}\Sigma_{s=0}^{t-2}k_{s} +\Sigma_{s=0}^{r-2}k_{s}, $$
and $ W^{4}:=\hspace{0.2cm} \{  k_{j}\geq 2 \quad \text{if} \quad 1\leq j \leq r-2, \hspace{0.2cm} k_{r-1}=n-2-\Sigma_{t=1}^{r-2}k_{t}\geq 2 \}$.\smallskip

Then $k_{1}\in [2,n-4-\Sigma_{t=2}^{r-2}k_{t}] $. \smallskip

Computing the second derivative of $f^{4}_{r}$ with respect to  $k_{1}$ we get:
$$\frac{\partial^2 f^{4}_{r}}{\partial k_{1}^2} = -2<0.$$

That is, the function is convex and the minimum is reached at the endpoints of the interval,  $k_{1}=2$ or  $k_{1}=n-4-\Sigma_{t=2}^{r-2}k_{t}$, i.e., $k_{1}=2$ or  $k_{r-1}=2$.\smallskip

By iterating this argument one can check that if $x^{*}=(k_{1},k_{2},\dots,k_{r-1}) $ satisfies $f^{3}_{r}(x^{*})
= \Delta_{r}$, then $k_{j}=2$ except for one $ j_{0}$ with $ 1\leq j_{0} \leq r-1$, and $ k_{j_{0}}=n-2r+2$. By symmetry, the cases $j_{0}=1$ and $ j_{0}=r-1$ provide the same value; furthermore, the cases $1< j_{0}< r-1$ provide the same value.\smallskip

If $j_{0}=1$ or $j_{0}=r-1$, then
\begin{eqnarray*}
f_{r}^{3}(x^{*})
&=& \frac{1}{2}n^{2}-1-n+2r-2-2-\frac{1}{2} (n-2r+2)^{2}\\
&&-\frac{1}{2}4(r-2)-2(n-2r+2)-4(r-3)\\
&=& n(2r-5)-2r^{2}+4r+5.
\end{eqnarray*}

If  $ 1<j_{0}<r-1$, substituting  $x^{*}=(2,\dots,2, n-2r+2,2,\dots,2)$ in $ f_{r}^{3}$ we get
\begin{eqnarray*}
f_{r}^{3}(x^{*})
&=& \frac{1}{2}n^{2}-5-\frac{1}{2} (n-2r+2)^{2}-\frac{1}{2}4(r-2)-4(n-2r+2)-4(r-4)\\
&=& 2n(r-3)-2r^{2}+6r+5. 
\end{eqnarray*}

Then  $\Delta_{r}=2n(r-3)-2r^{2}+6r+5$ for $r\geq 4$, since $n\geq 2r$.\smallskip

Note that if $ n=2r$, then $\Delta_{r}=2r^{2}-6r+5$, for every $r \geq 2$.
\end{proof}

The following result can be found in \cite{RSVV}. 

\begin{thm}\label{t:diam-delta}
 In any graph $G$ the inequality $\delta(G) \leq \frac{1}{2} \diam G$ holds. 
\end{thm}

We say that a vertex $v$ of a graph $G$ is a \emph{cut-vertex} if
$G \setminus \{v\}$ is not connected. A graph is \emph{two-connected} if it is connected and it does not contain cut-vertices.\smallskip

Given a graph $G$, we say that a family of subgraphs $\{G_{s}  \}  $ of $G$ is a \emph{T-decomposition} of $G$ if $\cup G_{s} = G $ and $G_{s}\cap G_{r}$   is either a \emph{cut-vertex} or the empty set for each $s\neq r$. Every graph has a \emph{T-decomposition}, as the following example shows. Given any edge in $G$, let us consider the maximal two-connected
subgraph containing it. We call to the set of these maximal \emph{two-connected} subgraphs $\{G_s\}_s$
the \emph{canonical T-decomposition} of $G$.\smallskip

Note that every $G_s$ in the \emph{canonical T-decomposition} of $G$
is an isometric subgraph of $G$.\smallskip


Given a graph $G$, let $\{G_{s}  \}  $ be the canonical T-decomposition of $G$. We define the  \emph{effective diameter} as:
$$
\diameff V(G):= \sup_s \text{diam}\,V(G_{s}), \hspace{0.3cm}\diameff G:= \sup_s \text{diam}\,G_{s}.
$$

The following result appears in \cite [Theorem 3]{BRSV2}.
\begin{lem} \label{l:bermu}
  Let $ G $ be a graph and $\{G_{s}  \}  $ be any T-decomposition of $G$, then $\delta(G) =\sup_{s} \delta (G_{s}) $.
 \end{lem}

We will need the following result,
which allows to obtain global information about the
hyperbolicity of a graph from local information (see Lemma \ref{l:bermu} and Theorem \ref{t:diam-delta}).

\begin{lem}\label{l:effdiam}
Let $G$ be any graph. Then
$$ \delta(G)\leq \frac{1}{2}\diameff (G) $$
\end{lem}

We define $ M(n,r):=\displaystyle{n \choose 2}-\Delta_{r}$, for $2\leq r\leq n/2$. 

We have the following expression for $M(n,r)$:

$$ M(n,2)=\dfrac{1}{2} [n^{2}-n-2]. $$

$$ M(n,3)=\dfrac{1}{2} [n^{2}-3n+2]. $$

$$ M(n,r)=\dfrac{1}{2} [(n-2r+3)^{2}+5n-19],\quad \text{if}\quad r\geq4. $$

\begin{lem}\label{l:effdiam2}
If $G\in \mathcal{G}(n,m)$ and $\diameff V(G)= \diam V(G)=r$, then $m\leq M(n,r).$
\end{lem}
\begin{proof}
Let us consider $u,v \in V(G)$ such that $d(u,v)= \text{diam}\,V(G)=r$.
Denote by $k_{j}$ the cardinal of $S_{j}:=\{ w\in V(G) \mid d(w,u)=j \}$ for $0\leq j \leq r$. Since  $\diameff V(G)= \text{diam}\,V(G)=r$, we have $k_{0}=1, k_{j}\geq 2$ for $1\leq j \leq r-1$, and $k_{r}\geq 1$.\smallskip

Note that a vertex of $S_{j}$ and a vertex of $S_{0}\cup S_{1} \cup \dots \cup S_{j-2}$ can not be neighbours for $2\leq j \leq r$. Denote by $x$ 
the minimum number of edges that can be removed from the complete graph with $n$ vertices in order to obtain $G$. Since the diameter of $V(G)$ is $r$, we have obtained the following lower bound for $x$:\\
$$x\geq k_{2} +k_{3}(1+k_{1})+ k_{4}(1+k_{1}+k_{2})+...+k_{r-1}(1+k_{1}+k_{2}+...+k_{r-3})$$ $$+k_{r}(1+k_{1}+k_{2}+...+k_{r-2})=f_{r} .$$

Then $ x\geq \Delta_{r}$ by Lemma \ref{l:lemma 1} and $m= \displaystyle{n \choose 2}-x \leq  \displaystyle{n \choose 2}-\Delta_{r} = M(n,r)$.
\end{proof}

\begin{lem}\label{l:lemma 2}
The inequality
$$ \displaystyle{n-n_{0}+1\choose 2} \leq M(n,r)-M(n_{0},r)$$ holds for  $2\leq r \leq  n_{0} /2 $ and $n>n_{0}.$
\end{lem}
\begin{proof}
If $r\geq 4$, then the inequality holds if and only if
$$\frac{1}{2}(n-n_{0}+1)(n-n_{0}) \leq \frac{1}{2}n(n-1) -\frac{1}{2}n_{0}(n_{0}-1)- 2(n-n_{0})(r-3)$$
$$\Leftrightarrow \quad  (n-n_{0}+1)(n-n_{0})\leq n^{2}-n_{0}^{2}-(n-n_{0})-4(n-n_{0})(r-3)$$
$$\Leftrightarrow \quad n-n_{0}+1\leq n+n_{0}-1-4(r-3)\quad\Leftrightarrow \quad 2r\leq n_{0}+5 , $$
and this holds since $2r\leq n_{0}$.\smallskip

If $r=3$, then 
$$\frac{1}{2}(n-n_{0}+1)(n-n_{0}) \leq \frac{1}{2}n(n-1) -\frac{1}{2}n_{0}(n_{0}-1)- (n-1)-(n_{0}-1)$$
$$\Leftrightarrow \quad  (n-n_{0}+1)(n-n_{0})\leq n^{2}-n_{0}^{2}-(n-n_{0})-2(n-n_{0})$$
$$\Leftrightarrow \quad n-n_{0}+1\leq n+n_{0}-3\quad\Leftrightarrow \quad n_{0}\geq 2, $$
and this holds since $n_{0} \geq 2r=6$.\smallskip

If $r=2$, then 
$$\frac{1}{2}(n-n_{0}+1)(n-n_{0}) \leq \frac{1}{2}n(n-1) -\frac{1}{2}n_{0}(n_{0}-1)\quad\Leftrightarrow \quad n-n_{0}+1\leq n+n_{0}-1\quad\Leftrightarrow \quad n_{0}\geq 1. $$
\end{proof}

\begin{lem}\label{l:effdiam3}
If $ G \in \mathcal{G}(n,m)$ and $\diameff V(G)= r$, then $m\leq M(n,r).$
\end{lem}

\begin{proof}
Given a graph $ G \in \mathcal{G}(n,m)$ with \emph{canonical T-decomposition} $ \{G_{s}\}$, let $G_{k}$ be a subgraph with $\diameff V (G_{k})=\diameff V(G) =r$. If $G_{k}$ has  $n_{0}$ vertices and  $m_{0}$ edges, then $ m_{0}\leq M(n_{0},r)$ by Lemma \ref{l:effdiam2}. Note that $2r\leq n_{0}$.\smallskip

Completing $G_{k}$ with the complete graph of  $n-n_{0}+1$ vertices (one of the vertices belongs to $G_{k}$) we get that $m\leq m_{0}+ \displaystyle{n-n_{0}+1\choose 2} .$\smallskip

By Lemma \ref{l:lemma 2} we have $ m\leq m_{0}+M(n,r)-M(n_{0},r)$ and, since $m_{0}\leq M(n_{0},r)$, we conclude $ m\leq  M(n,r). $ 
\end{proof}
\begin{cor}\label{cor:effdiam}
 If $ G \in \mathcal{G}(n,m) $, $2\leq r \leq n/2 $ and $ m> M(n,r)$, then $\diameff V (G)\neq r $.
\end{cor}

We will show now that, in fact, this result can be improved.
\begin{thm}\label{t:effdiam}
 If $ G \in \mathcal{G}(n,m)$, $2\leq r \leq n/2 $ and $m> M(n,r)$, then $\diameff V(G)< r. $
\end{thm}

\begin{proof}
By Corollary \ref{cor:effdiam}, it suffices to prove that $M(n,r)$ is a decreasing function  of $ r $.
We have $\Delta_{2} \leq \Delta_{3} \leq \Delta_{4} $, since $1\leq n-1 \leq 2n-3$. Thus,  $M(n,2) \geq M(n,3) \geq M(n,4) $.
If $r \geq 4$, then $M(n,r)$ decreases as a function  of $ r $ since   $2r\leq n$ gives $ \dfrac{\partial M(n,r)}{\partial r}=-2(n-2r+3)\leq0$.

\end{proof}

Since $\diameff V(G)<r$ implies  $\diameff G \leq r $, Lemma \ref{l:effdiam} and Theorem \ref{t:effdiam} imply the following theorems.
\begin{thm}\label{t:bound1}
 If $  G \in \mathcal{G}(n,m) $, $ 2\leq r \leq n/2 $ and
$ m > M(n,r)$, then $\delta(G)\leq r/2  .$ 
\end{thm}

Define $M(n,1):=n(n-1)/2.$

\begin{thm}\label{t:bound2}
 
 If $n\geq1$ and $m=n-1$, then $B(n,m)=0$. If $n\geq3$ and $n\leq m \leq n+3$, then $B(n,m)=n/4$.
 If $  G \in \mathcal{G}(n,m) $, $ 2\leq r \leq n/2 $ and
$  M(n,r)<m \leq M(n,r-1)$, then $B(n,m)\leq r/2 $. 
\end{thm}

\begin{proof}

If $n\geq1$ and $m=n-1$, then every  $G \in \mathcal{G}(n,m) $ is a tree and $\delta(G)=0$; consequently, $B(n,m)=0$.\smallskip 

If $n\geq3$ and $n\leq m\leq n+3$, then  \cite [Theorem 30]{MRSV} gives that there exists $G_{0} \in \mathcal{G}(n,m) $ with $\delta(G_{0})=n/4$. Furthermore, $\delta(G)\leq n/4$ for every $n,m$ and $G \in \mathcal{G}(n,m) $ by \cite [Theorem 30]{MRSV}. Hence, $B(n,m)=n/4$ for $3\leq n \leq m \leq n+3$.\smallskip 

The second part of the statement is a consequence of Theorem \ref{t:bound1}.

 \end{proof}

\section{A lower bound for $B(n,m)$}

\begin{thm}\label{lowerbound0}
  If $3\leq n_{0} \leq n$ and
  $ n < m \leq n+ \displaystyle{n_{0} -1\choose 2},$
  then $B(n,m)\geq  (n-n_{0}+3)/4.$
\end{thm}
\begin{proof}
Let us consider a cycle graph with $n$ vertices $C_{n}$. Given $n_{0}\geq 3$, choose a path $\{v_{1},...,v_{n_{0}}\}$ in $C_{n}$ and add $\displaystyle{n_{0}\choose 2} -(n_{0}-1)=\displaystyle{n_{0}-1\choose 2}$ edges to  $C_{n}$ if $n_{0}<n$, or $\displaystyle{n\choose 2} -n$ if $n_{0}=n$, obtaining a graph $G_{n,n_{0}}$ such that the induced subgraph by $\{v_{1},...,v_{n_{0}}\}$ in $G_{n,n_{0}}$ is isomorphic to the complete graph with $n_{0}$ vertices.\smallskip

Choose a path  $\{v_{1},...,v_{n_{0}}\}$ in $C_{n}$ and add $m-n$ edges to  $C_{n}$, obtaining a subgraph $G$ of $G_{n,n_{0}}$ with at least some $v_{i}$ verifying $[v_{i},v_{1}],[v_{i},v_{n_{0}}]\in E(G)$. If $n_{0}=3$, then $n<m\leq n+1$ and $m=n+1.$\smallskip

Note that $G \in \mathcal{G}(n,m) $. Let $\eta$ be the path in $C_{n}$ joining $v_{1}$ and $v_{n_{0}}$  with $v_{2},...,v_{n_{0}-1}\notin \eta$  and let $y$ be the midpoint of $\eta$. Define $x:=v_{i}$, $\gamma_{1}=[x,v_{1}]\cup[v_{1} y]$ and $\gamma_{2}=[x,v_{n_{0}}]\cup[v_{n_{0}} y]$. Then $\gamma_{1}$ and $\gamma_{2}$ are geodesics from $x$ to $y$ and $$d_{G}(x,y)=1+\frac{n-(n_{0}-1)}{2}=\frac{n-n_{0}+3}{2}.$$

 Consider the geodesic bigon $T=\{\gamma_{1},\gamma_{2}\}$ and the midpoint $p$ of $\gamma_{1}$. Then $$B(n,m)\geq \delta(G)\geq d_{G}(p,\gamma_{2})=\frac{1}{2}L(\gamma_{1})=\frac{n-n_{0}+3}{4}.$$ 
 \end{proof}
Theorems \ref{t:bound2} and \ref{lowerbound0}  have the following direct consequence.

\begin{thm}\label{lowerbound}

If $n\geq1$ and $m=n-1$, then $B(n,m)=0$. If $n\geq3$ and $n\leq m \leq n+3$, then $B(n,m)=n/4$. If $5\leq n_{0} \leq n$ and
  $ n + \displaystyle{n_{0} -2\choose 2}  < m \leq n+ \displaystyle{n_{0} -1\choose 2},$
    then $B(n,m)\geq  (n-n_{0}+3)/4.$

\end{thm}

\section{Difference of the bounds of $B(n,m)$}


 Let $b_{1}(n,m)$ and $b_{2}(n,m)$ be the lower and  upper bounds of $B(n,m)$ obtained in  Theorems \ref{lowerbound} and \ref{t:bound2}, respectively. In this section we prove that the difference between $b_{2}$ and $b_{1}$ is $O(\sqrt{n}\, )$. This is a good estimate, since the sharp upper bound for graphs with $n$ vertices is $n/4$ (see \cite [Theorem 30]{MRSV}). \smallskip

\begin{lem}\label{estcot}
   Given integers $n$ and $r$ with   $2\leq r\leq n/2$, let $n_{0}$ be the smallest natural number such that  $3\leq n_{0}\leq n$ and $M(n,r)<n+\displaystyle{n_{0} -1\choose 2}$. 
   The following holds for  $M(n,r)<m \leq n+\displaystyle{n_{0} -1\choose 2}$.
 
  \begin{itemize}

    \item  If $r=2$, then  $ b_{2}(n,m)=b_{1}(n,m).$

   \item  If $r=3$, then $ b_{2}(n,m)-b_{1}(n,m)< 3/4.$

    \item  If $4\leq r \leq n/2$, then $b_{2}(n,m)-b_{1}(n,m)  < \sqrt{3n}/4.$
   
    
  \end{itemize}

\end{lem}

\begin{rem}
Note that we always have $M(n,r) \leq \frac{1}{2}n(n-1)< n+\displaystyle{n -1\choose 2}$, and this implies the existence of $n_{0}$.
\end{rem}

\begin{proof}


 If $r=2$, then $M(n,r)=\displaystyle{n \choose 2}-1  $ and $M(n,2)<m$ implies $m=\displaystyle{n \choose 2} $. Hence, every graph $G \in \mathcal{G}(n,m) $ is isomorphic to the complete graph with $n$ vertices, and  $\delta(G)=1$ since $n\geq4$. Thus, $A(n,m)=B(n,m)=1$ and $ b_{1}(n,m)=b_{2}(n,m)=1.$\smallskip

 If $r=3$, then
 \smallskip

 $$M(n,3)<  n+\displaystyle{n_{0} -1\choose 2}\quad\Leftrightarrow \quad   \displaystyle{n\choose 2}-(n-1)<n+\displaystyle{n_{0} -1\choose 2}\quad\Leftrightarrow \quad n^{2}-5n< n_{0}^{2}-3n_{0}.$$\smallskip

 Let us define $\lambda_{3} := n^{2}-5n$. Since $n\geq4$, the smallest $n_{0}$ verifying the previous inequality is the smallest $n_{0}$ satisfying $n_{0}> \dfrac{3+\sqrt{9+4\lambda_{3}}}{2}$. Thus $n_{0}\leq  \dfrac{5+\sqrt{9+4\lambda_{3}}}{2}=:n_{0}' $.
 
  Then, the following  holds
  $$\frac{r}{2}- \frac{n-n_{0}+3}{4}\leq \dfrac{3}{2}- \frac{n-n_{0}'+3}{4}= \dfrac{ 11-2n+  \sqrt{9+4\lambda_{3}}}{8}=\dfrac{ 11-2n+  \sqrt{4(n-5/2)^{2}-16}}{8}.$$
   
   Note that
   
    $$  \dfrac{ 11-2n+\sqrt{4(n-5/2)^{2}-16}   }{8} <  \dfrac{ 11-2n+\sqrt{4(n-5/2)^{2}}   }{8} = \dfrac{ 11-2n+2(n-5/2) }{8} =\dfrac{3}{4}. $$

  Therefore, for $r=3$, we obtain $ b_{2}(n,m)-b_{1}(n,m)=\dfrac{r}{2} -\dfrac{n-n_{0}+3}{4}  <\dfrac{3}{4}.$
  \smallskip

 Note that if $r\geq4$, then
 \smallskip

 $$M(n,r)<  n+\displaystyle{n_{0} -1\choose 2}\quad\Leftrightarrow \quad   \displaystyle{n\choose 2}-(2n(r-3)-2r^{2}+6r+5)<n+\displaystyle{n_{0} -1\choose 2}$$
 $$\Leftrightarrow \quad n^{2}+9n-4nr+4r^{2}-12r-12< n_{0}^{2}-3n_{0}.$$\smallskip

 Let us define $\lambda_{r} := n^{2}+9n-4nr+4r^{2}-12r-12$. Then, the smallest $n_{0}$ verifying the previous inequality is the smallest $n_{0}$ satisfying $n_{0}> \dfrac{3+\sqrt{9+4\lambda_{r}}}{2} $. Thus $n_{0}\leq \dfrac{5+\sqrt{9+4\lambda_{r}}}{2}=:n_{0}' $.

 Note that
  $$\frac{r}{2}- \frac{n-n_{0}+3}{4}\leq \frac{r}{2}- \frac{n-n_{0}'+3}{4}= \dfrac{ 4r + \sqrt{9+4\lambda_{r}}-2n-1}{8}.$$

  Let us fix $n$ and consider the function $F(r)=4r+\sqrt{9+4\lambda_{r}}$. It can be easily checked that $F'(r)=4+\dfrac{2(-4n+8r-12)}{\sqrt{9+4\lambda_{r}}}>0$ for all $r\in [4,n/2]$ if and only if $n>6$. 
   
 Since $r\geq4$, we have $n\geq8$, $F(r)$ is an increasing function and  $F(n/2)=2n+\sqrt{9+4(3n-12)}$ is the maximum value of $F(r)$.

  Then, the following inequalities hold
 $$ b_{2}(n,m)-b_{1}(n,m)=\frac{r}{2}- \frac{n-n_{0}+3}{4}\leq \dfrac{F(n/2)-2n -1}{8}< \dfrac{\sqrt{9+4(3n-12)} }{8}< \dfrac{2\sqrt{3n}}{8}=\dfrac{\sqrt{3n}}{4}. $$ 
 
\end{proof}

\begin{lem} \label{diff}

 Given  integers $n$ and $r$ with   $3\leq r\leq n/2$, let $n_{1}$ be the smallest natural number such that $3\leq n_{1}\leq n$ and $M(n,r-1)< n+\displaystyle{n_{1} -1\choose 2}$. Consider $n_{0}$ defined as in Lemma \ref{estcot}. The following holds.

   \begin{itemize}

    \item  If $r=3$, $r=4$ or  $r= n/2$, then $n_{1}-n_{0}\leq2 .$

     \item  If $5\leq r < n/2$, then $ n_{1}-n_{0}\leq4 .$

   \end{itemize}

\end{lem}

 \begin{proof}

  If $r=3$, then
  \smallskip

  $$M(n,r-1)<  n+\displaystyle{n_{1} -1\choose 2}\quad\Leftrightarrow \quad n^{2}-3n-4< n_{1}^{2}-3n_{1}.$$\smallskip

  Using the definition of $\lambda_{r}$ in the proof of Lemma \ref{estcot}, we deduce that  the smallest natural number $n_{1}$ verifying the previous inequality satisfies $n_{1}\leq  \dfrac{5+\sqrt{9+4\lambda_{2}}}{2}=:n_{1}' $.\smallskip
 
 If $r=4$, then
  \smallskip
  
  $$M(n,r-1)<  n+\displaystyle{n_{1} -1\choose 2}\quad\Leftrightarrow \quad n^{2}-5n< n_{1}^{2}-3n_{1}.$$\smallskip

  Therefore, the smallest $n_{1}$ verifying the previous inequality satisfies $n_{1}\leq \dfrac{5+\sqrt{9+4\lambda_{3}}}{2}=:n_{1}' $.\smallskip
 
 Note that if $r\geq5$, then
 \smallskip

 $$M(n,r-1)<  n+\displaystyle{n_{1} -1\choose 2}$$
 $$\Leftrightarrow \quad n^{2}+13n-4nr+4r^{2}-20r+4=n^{2}+9n-4n(r-1)+4(r-1)^{2}-12(r-1)-12< n_{0}^{2}-3n_{0}.$$\smallskip


Thus, the smallest $n_{1}$ verifying the previous inequality satisfies $n_{1}\leq \dfrac{5+\sqrt{9+4\lambda_{r-1}}}{2}=:n_{1}' $.\smallskip

Now we estimate the difference between $n_{1}$ and $n_{0}$.\smallskip

$$n_{1}-n_{0}<n_{1}'-(n_{0}'-1)=\dfrac{\sqrt{9+4\lambda_{r-1}}-\sqrt{9+4\lambda_{r}}     }{2}+1 = \dfrac{2(\lambda_{r-1}-\lambda_{r})}{\sqrt{9+4\lambda_{r-1}}+\sqrt{9+4\lambda_{r}} }+1\leq \dfrac{\lambda_{r-1}-\lambda_{r}}{\sqrt{9+4\lambda_{r}} }+1. $$

If $r=3$, then  $n\geq6$ and

$$n_{1}-n_{0}< \dfrac{\lambda_{2}-\lambda_{3}}{\sqrt{9+4\lambda_{3}} }+1= \dfrac{2(n-2)}{\sqrt{9+4\lambda_{3}}}+1<\dfrac{n-2}{\sqrt{\lambda_{3}}}+1.$$

The following holds

$$ \dfrac{n^{2}-4n+4}{n^{2}-5n}=\dfrac{n^{2}-5n+n+4}{n^{2}-5n} <3 \quad  \Rightarrow  \quad \dfrac{n-2}{\sqrt{\lambda_{3}}}<\sqrt{3} .$$

Therefore,

$$n_{1}-n_{0}<\sqrt{3}+1 \quad  \Rightarrow  \quad  n_{1}-n_{0}\leq 2. $$

If $r=4$, then $n\geq8$ and

$$n_{1}-n_{0}< \dfrac{\lambda_{3}-\lambda_{4}}{\sqrt{9+4\lambda_{4}} }+1=  \dfrac{2(n-2)}{\sqrt{9+4\lambda_{4}}}+1<\dfrac{n-2}{\sqrt{\lambda_{4}}}+1.$$

The following holds 

$$ \dfrac{n^{2}-4n+4}{n^{2}-7n+4}=\dfrac{n^{2}-7n+4+3n}{n^{2}-7n+4} \leq3  \quad  \Rightarrow  \quad      \dfrac{n-2}{\sqrt{\lambda_{4}}} \leq \sqrt{3}.$$

Therefore,

$$n_{1}-n_{0}<\sqrt{3}+1 \quad  \Rightarrow  \quad  n_{1}-n_{0}\leq 2. $$

If $r\geq5$, then

$$n_{1}-n_{0}< \dfrac{\lambda_{r-1}-\lambda_{r}}{\sqrt{9+4\lambda_{r}} }+1=\dfrac{4(n-2r+4)}{\sqrt{9+4\lambda_{r}}} +1.$$


Note that

$$\lambda_{r} = (n-2r)^{2}+9(n-2r)+6r-12\geq(n-2r)^{2}+6r-12\geq(n-2r)^{2}+18. $$

If $r< n/2$, then

$$\dfrac{4(n-2r+4)}{\sqrt{9+4\lambda_{r}} }\leq\dfrac{4(n-2r+4)}{\sqrt{81+4(n-2r)^{2}}}<2\dfrac{n-2r}{n-2r}+\dfrac{16}{9}<4.$$

Thus, $n_{1}-n_{0}<5$ and $n_{1}-n_{0}\leq4$.\smallskip

If $r= n/2$, then 

$$ \dfrac{4(n-2r+4)}{\sqrt{9+4\lambda_{r}} } \leq \dfrac{4(n-2r+4)}{\sqrt{81+4(n-2r)^{2}}}=\dfrac{16}{9}<2.$$

Therefore $n_{1}-n_{0}<3$ and $n_{1}-n_{0}\leq2$.
\end{proof}

The following result is a consequence of the two previous lemmas.

\begin{lem} \label{finaldiff}
   Given  integers $n$ and $r$ with    $3\leq r\leq n/2$, let $n_{0}$ be defined as in Lemma \ref{estcot}. Assume $M(n,r-1)> n+\displaystyle{n_{0} -1\choose 2}$. The following holds for $n+\displaystyle{n_{0} -1\choose 2}<m\leq M(n,r-1)$.

    \begin{itemize}

    \item 
     If $r=3$, then $ b_{2}(n,m)-b_{1}(n,m)< 5/4 .$

   \item  If $r=4$ or $r=n/2$, then $ b_{2}(n,m)-b_{1}(n,m)< \sqrt{3n}/4 +1/2.$

    \item If $5\leq r < n/2$, then $ b_{2}(n,m)-b_{1}(n,m)< \sqrt{3n}/4 + 1.$    
     \end{itemize}

  \end{lem}

\begin{proof}


Let $n_{1}$ be defined as in Lemma \ref{diff}. 



On the other hand, $m\leq  M(n,r-1) <n+ \displaystyle{n_{1} -1\choose 2}$ and Theorem \ref{lowerbound0} gives $ b_{1}(n,m)\geq(n-n_{1}+3)/4$.

On the other hand, $M(n,r) < n+ \displaystyle{n_{0} -1\choose 2}< m\leq M(n,r-1)$ and Theorem \ref{t:bound2} gives $b_{2}(n,m)= r/2$.

The following holds

$$ b_{2}(n,m)-b_{1}(n,m) =b_{2}(n,m)-\dfrac{n-n_{0}+3}{4} + \dfrac{n-n_{0}+3}{4} - b_{1}(n,m).$$

Notice that 

$$\dfrac{n-n_{0}+3}{4} - b_{1}(n,m)\leq\dfrac{n-n_{0}+3}{4}-\dfrac{n-n_{1}+3}{4}=\dfrac{n_{1}-n_{0}}{4}.$$

Then, applying Lemmas  \ref{estcot} and \ref{diff}, in order to bound $b_{2}(n,m)-(n-n_{0}+3)/4$ and $n_{1}-n_{0}$, respectively, we obtain the desired upper bounds.
\end{proof}

Lemmas \ref{estcot} and \ref{finaldiff} have the following consequence.

\begin{thm} \label{t:final}
The following holds for all $n\geq 3$.

     \begin{equation} \label{eq:1}
      b_{2}(n,m)-b_{1}(n,m)< \frac{\sqrt{3n}}{4}+1.
     \end{equation}

  \end{thm}

\begin{proof}  
    
    If $m>M(n,3)$, then $b_{2}(n,m)\leq 3/2$ by Theorem
    \ref{t:bound1}, and 
    $$    b_{2}(n,m)-b_{1}(n,m)\leq  b_{2}(n,m)\leq \dfrac{3}{2}< \dfrac{3}{4}+1\leq \dfrac{\sqrt{3n}}{4} +1.      $$

  Consider now $r\geq3$ and $n_{0}$ defined as in Lemma \ref{estcot}. If $M(n,r)<m \leq n+\displaystyle{n_{0} -1\choose 2}$ or $M(n,r-1)<m \leq n+\displaystyle{n_{1} -1\choose 2}$, then Lemma  \ref{estcot} gives

       \begin{equation} \label{eq:2}
       b_{2}(n,m)-b_{1}(n,m)  < \dfrac{\sqrt{3n}}{4}.
       \end{equation}

  If  $M(n,r-1)\leq n+\displaystyle{n_{0} -1\choose 2}$, then  equation \ref{eq:2} holds for $M(n,r)<m \leq M(n,r-1)$.\smallskip

  If  $n+\displaystyle{n_{0} -1\choose 2}< M(n,r-1)$ and $n+\displaystyle{n_{0} -1\choose 2}<m\leq M(n,r-1)$, then Lemma   \ref{finaldiff} implies \ref{eq:1}. Thus, equation \ref{eq:1} holds for $M(n,r)<m \leq M(n,r-1)$.\smallskip
  
  Hence, \ref{eq:2} holds for every $m>M(n,\lfloor n/2 \rfloor)$.\smallskip
  
  Finally, assume that $n+3<m\leq M(n,\lfloor n/2 \rfloor) $.\smallskip
  
  First, note that if $M(n,\lfloor n/2 \rfloor)<m\leq \text{min} \hspace{0.7mm}  \Big\{ n+\displaystyle{n_{0} -1\choose 2},M(n,\lfloor n/2 \rfloor -1)  \Big\},$
  then Lemma \ref{estcot} implies

    $$    b_{2}(n,m)-b_{1}(n,m)= \dfrac{\lfloor \frac{n}{2} \rfloor}{2} -\dfrac{n-n_{0}+3}{4}   <\dfrac{\sqrt{3n}}{4} .      $$  
                
    Consider now $m\leq M(n,\lfloor n/2 \rfloor)$, then
    $$ b_{2}(n,m)-b_{1}(n,m)\leq \frac{n}{4}   -\dfrac{n-n_{0}+3}{4} <    \dfrac{2(\lfloor \frac{n}{2} \rfloor +1) }{4} -\dfrac{n-n_{0}+3}{4} =  \dfrac{ \lfloor \frac{n}{2} \rfloor }{2} -\dfrac{n-n_{0}+3}{4} +\frac{1}{2} <\dfrac{\sqrt{3n}}{4} +\frac{1}{2} .      $$

    Hence, \ref{eq:2} holds for every $m\leq M(n,\lfloor n/2 \rfloor)$.
 
  \end{proof}

\section{Computation of  $A(n,m)$}

 Denote by $\Gamma_{3}$ the set of graphs such that every cycle has length $3$ and every edge belongs to some cycle.
 \begin{prop}\label{p:triangle}
   Consider a graph $G\in  \mathcal{G}(n,m)\cap \Gamma_{3} $. If $k$ denotes the number of cycles of $G$, then $n=2k+1$ and $m=3k$.
 \end{prop}
\begin{proof}
Let us prove the result by induction on $k$.\smallskip

If $k=1$, then $G$ is isomorphic to $C_{3}$ and $n=m=3$. \smallskip

Assume that the statement holds for every graph $G_{0}$ with $k-1$ cycles. Then $G_{0}$ has $n_{0}=2(k-1)+1$ vertices and $m_{0}=3(k-1)$ edges. Any graph $G$ with $k$ cycles can be  obtained by adding $2$ vertices and $3$ edges to some graph $G_{0}$ with $k-1$ cycles, that is, $n=n_{0}+2=2k+1$ and $m=m_{0}+3=3k$.
\end{proof}

We say that an edge $g$ of a graph $G$ is a \emph{cut-edge} if
$G \setminus \{g\}$ is not connected.
Given a graph $G$, the \emph{T-edge-decomposition} of $G$ is a \emph{T-decomposition} such that each component $G_{s}$ is either a \emph{cut-edge} or it does not contain cut-edges.\smallskip

 \begin{prop}\label{p:c.triangles}
Let $G\in  \mathcal{G}(n,m) $ be a graph such that every cycle has length $3$. Then $2m\leq 3n-3$.
 \end{prop}
 
\begin{proof}
 The canonical T-edge-decomposition  of $G$ has  $r\geq 1$ graphs $\{G_{1},...,G_{r}\}$ in $\Gamma_{3}$ and $s\geq 0$ edges $\{G_{r+1},...,G_{r+s}\}$.
 For each component $G_{i} \in \Gamma_{3}$ we have, by \ref{p:triangle},
$$ n_{i}=2k_{i}+1,\hspace{0.2cm}  m_{i}=3k_{i},\hspace{0.2cm} 1\leq i \leq r, $$
where $n_{i}$, $m_{i}$ and $k_{i}$  denote the number of vertices, edges and cycles in $G_{i}$, respectively.\smallskip

Let us denote by $ k= \Sigma_{i=1}^{r}k_{i}$ the number of cycles of $G$. Let $n_{0}$ and $m_{0}$ be the number of vertices and edges we add to complete $G$, i.e, $ n_{0}= n-\Sigma_{i=1}^{r}n_{i}$, $ m_{0}= m-\Sigma_{i=1}^{r}m_{i}$. Then we have
$$ n= \Sigma_{i=0}^{r}n_{i}=n_{0}+\Sigma_{i=1}^{r}(2k_{i}+1)=n_{0}+2k+r,$$
$$ m= \Sigma_{i=0}^{r}m_{i}=m_{0}+\Sigma_{i=1}^{r}(3k_{i})=m_{0}+3k.$$

Hence,
$$n=n_{0}+2\,\frac{m-m_{0}}{3}+r.$$

One can check that if $n_{0}=0$, then $m_{0}=r-1$ and if $n_{0}\geq 1$, then $m_{0}=n_{0}+r-1$. Therefore,
$$n=n_{0}+2\,\frac{m-(n_{0}+r-1)}{3}+r \hspace{0.2cm}\Rightarrow \hspace{0.2cm} 2m=3n-n_{0}-r-2\hspace{0.2cm}\Rightarrow \hspace{0.2cm} 2m=3n-m_{0}-3.$$

Then $2m\leq 3n-3$.
\end{proof}

The next result appears in \cite{MRSV}.
\begin{thm}
 \label{t:delta2} 
 Let $G$ be any graph.

\begin{itemize}

\item $\d(G)< 1/4$ if and only if $G$ is a tree.



\item $\d(G)< 1$ if and only if every cycle $g$ in $G$ has
length $L(g) \le 3$.

\end{itemize}
Furthermore, if $\d(G)< 1$, then $\d(G)\in \{0,3/4\}$.
\end{thm}

\begin{prop}\label{precise-val}
 If $m\geq n$ and $2m\leq 3n-3$, then $A(n,m)=3/4$.
\end{prop}
\begin{rem}
Note that $n\leq m\leq (3n-3) /2$ implies $n\geq 3$.
\end{rem}
\begin{proof}
Since $m\geq n\geq 3$, if $G\in  \mathcal{G}(n,m)$, then $G$ is not a tree. Hence Theorem \ref{t:delta2}
gives $\delta(G)\geq 3/4 $ and $A(n,m)\geq 3/4 $.\smallskip

Fix $n,m$ verifying the hypotheses. Define $n_{0}:=m_{0}:=3n-3-2m$ and $k:=m+1-n$. Then 
 $$n=2k+1+n_{0},\qquad  m=3k+n_{0}.$$
 
 Let us consider $k$ graphs $G_{1},\dots,G_{k}$ isomorphic to $C_{3}$ and $n_{0}$ graphs $\Gamma_{1},\dots,\Gamma_{n_{0}}$ isomorphic to $P_{2}$. Fix vertices $v_{1}\in V(G_{1}),\dots,v_{k}\in V(G_{k}),w_{1}\in V(\Gamma_{1}),\dots,w_{n_{0}}\in V(\Gamma_{n_{0}})$ and consider the grah $G$ obtained from $G_{1},\dots,G_{k},\Gamma_{1},\dots,\Gamma_{n_{0}}$ by identifying $v_{1},\dots,v_{k},w_{1},\dots,w_{n_{0}}$ in a single vertex. Then $G\in  \mathcal{G}(n,m)$ and $\delta (G)=3/4 $. Therefore, $A(n,m)\leq 3/4 $ and we conclude $A(n,m)=3/4 $.
\end{proof}

\noindent

\begin{defn}\label{def:c.Kn}
 Let $K_{n}$ be the complete graph with $n$ vertices and consider the numbers $N_{i}$, $i=1,\dots, s$, $(s\geq 1)$ such that $2\leq N_{1}, \dots, N_{s}< n$, $ N_{1}+ \dots+ N_{s}\leq n $. Choose sets of vertices $V_{1}, \dots, V_{s} \subset V(K_{n})$ with $V_{i} \cap  V_{j}= \emptyset$ if  $i\neq j$ and $\#V_{i}=N_{i}$ for $i=1,\dots, s$. Let $K_{n}^{N_{1},\dots, N_{s}}$ be the graph obtained from $K_{n}$ by removing the edges joining any two vertices in $V_{i}$ for every $i=1,\dots, s$.
\end{defn}

\begin{lem}\label{l:c.Kn1}
We always have $  \delta(K_{n}^{N_{1},\dots, N_{s}} ) \leq 1 $. 
\end{lem}
\begin{proof}
Fist of all, note that $\diam V(K_{n}^{N_{1},\dots, N_{s}})=2$. Hence, in order to prove $\diam (K_{n}^{N_{1},\dots, N_{s}})=2$, it suffices to check that $d(x,y)\leq 2$ for every midpoint $x$ of any edge in $E(K_{n}^{N_{1},\dots, N_{s}})$ and every $y\in K_{n}^{N_{1},\dots, N_{s}} $. \smallskip

Fix $i \in \{1,\dots, s\}$ and $u\in V_{i}$. Then, $d(u,v)=1$ for every $v\in V(K_{n}^{N_{1},\dots, N_{s}})\setminus V_{i}$  and $d(u,v)=2$ for every $v\in V_{i}\setminus \{u\}$. \smallskip

Given a fixed  vertex $u\in V_{i}$,  let $x$ be the midpoint of the edge $[u,v]$ (then $v \notin V_{i}$). If   $w\in V_{i}$, then there   exists an edge joining $v$ with $w$.  Therefore, we have $d(x,w)\leq d(x,v)+d(v,w)=3/2$. If $w\notin V_{i}$, then $ [u,w]  \in E(K_{n}^{N_{1},\dots, N_{s}})$ and $d(x,w)\leq d(x,u)+d(u,w)=3/2$. Hence, $d(x,v)\leq 3/2$ for every $v\in V(K_{n}^{N_{1},\dots, N_{s}})$; thus, $d(x,y)\leq2$ for every
$y \in K_{n}^{N_{1},\dots, N_{s}}$.\smallskip

If $N_{1}+\dots+ N_{s} \leq n-2$, let $x$ be the midpoint of $[v^{1},v^{2}]$, where $v^{1},v^{2}\notin \cup _{i}V_{i}$. If $v\in V(K_{n}^{N_{1},\dots, N_{s}})$, then there exists an edge joining $v$ with $v^{1}$. Thus, we have  $d(x,v)\leq d(x,v^{1})+d(v^{1},v)=3/2$ for every $v\in V(K_{n}^{N_{1},\dots, N_{s}})$. Hence, $d(x,y)\leq2$ for every
$y \in K_{n}^{N_{1},\dots, N_{s}}$.
\smallskip

Therefore $\diam (K_{n}^{N_{1},\dots, N_{s}})=2$ and $  \delta(K_{n}^{N_{1},\dots, N_{s}} ) \leq 1 $ by Theorem \ref{t:diam-delta}.
\end{proof}

 In order to prove our next result we need the following Combinatorial lemma.

\begin{lem}\label{l:comb}
 For all $t\geq 3$, $(t\neq4,5)$, there exist numbers  $t_{i}\geq2$, $i=1,\dots, s$ $(s\geq 1)$, such that 
 $$ \Sigma_{i} t_{i} \leq t  \quad \text{and} \quad    \Sigma_{i} \displaystyle{t_{i} \choose 2} = t.$$
\end{lem}

\begin{proof}
If $t=3$, then choose $t_{1}=3$, $3\leq3$ and $\displaystyle{3 \choose 2} = 3$.\smallskip

If $t=6$, then  choose $t_{1}=4$, $4\leq6$ and $\displaystyle{4 \choose 2} = 6$.\smallskip

If $t=7$, then choose  $t_{1}=4$, $t_{2}=2$, $4+2\leq7$ and $\displaystyle{4 \choose 2}+ \displaystyle{2 \choose 2}= 7$.\smallskip

If $t=8$, then choose  $t_{1}=4$, $t_{2}=2$, $t_{3}=2$, $4+2+2\leq8$ and $\displaystyle{4 \choose 2} + \displaystyle{2 \choose 2}+\displaystyle{2 \choose 2}= 8$.\smallskip

If $t=9$, then choose  $t_{1}=4$, $t_{2}=3$, $4+3\leq9$ and $\displaystyle{4 \choose 2}+\displaystyle{3 \choose 2} = 9$.\smallskip

Let us prove the result by induction on $t$.\smallskip


We have seen that 

$$ \Sigma_{i} t_{i} \leq t,  \quad   \Sigma_{i} \displaystyle{t_{i} \choose 2} = t$$

holds for $6\leq t \leq9$. Assume now that it holds for every value $3,6,7,\dots,t-1$, with $t>9$. Then it holds for $t-3\geq6$ and there exist numbers  $t_{i}\geq2$, $i=1,\dots, s$,  such that $ \Sigma_{i} t_{i} \leq t-3$ and $ \Sigma_{i} \displaystyle{t_{i} \choose 2} = t-3$.\smallskip

Therefore, there exist numbers  $t_{i}'\geq2$, $t_{i}'= t_{i}$ for $i=1,\dots, s$, $t_{s+1}'=3$ such that
 
$$ \Sigma_{i} t_{i}'=\Sigma_{i} t_{i}+3 \leq t$$ 
and 
$$ \Sigma_{i} \displaystyle{t_{i}' \choose 2} =\Sigma_{i} \displaystyle{t_{i} \choose 2}+\displaystyle{3 \choose 2}= t.$$.\smallskip

So we have shown that the statement holds at $t$ when it is assumed to be true for $3,6,7,\dots ,t-1$.
\end{proof}

\begin{cor}\label{c:comb}
 For all $t\geq 1$, there exist numbers  $t_{i}\geq2$, $i=1,\dots, s$, $(s\geq 1)$ such that  $ \Sigma_{i} t_{i} \leq t+2$ and $ \Sigma_{i} \displaystyle{t_{i} \choose 2} = t$. 
\end{cor}

\begin{proof}

If $t\neq 1,2,4,5$, then Lemma \ref{l:comb} gives the result.\smallskip

If $t=1$, then choose $t_{1}=2$, $2\leq3$ and $\displaystyle{2 \choose 2} = 1$.\smallskip

If $t=2$, then  choose $t_{1}=2$, $t_{2}=2$, $2+2\leq4$ and $\displaystyle{2 \choose 2} +\displaystyle{2 \choose 2} = 2$.\smallskip

If $t=4$, then choose  $t_{1}=3$, $t_{2}=2$, $3+2\leq6$ and $\displaystyle{3 \choose 2}+ \displaystyle{2 \choose 2}= 4$.\smallskip

If $t=5$, then choose  $t_{1}=3$, $t_{2}=2$, $t_{3}=2$,  $3+2+2\leq7$ and $\displaystyle{3 \choose 2} + \displaystyle{2 \choose 2}+\displaystyle{2 \choose 2}= 5$.\smallskip

\end{proof}


\begin{prop}\label{p:bounds}
 If $m\geq n$ and $2m> 3n-3$, then $ A(n,m)=1$.
\end{prop}

\begin{proof}

Consider any $G\in  \mathcal{G}(n,m)$. Proposition \ref{p:c.triangles} gives that there exists at least one cycle in $G$ with length greater or equal than $4$. Then Theorem \ref{t:delta2} gives $\delta(G)\geq1$ for every $G\in  \mathcal{G}(n,m)$ and, consequently, $A(n,m)\geq1$.\smallskip

In order to finish the proof it suffices to find a graph $G\in  \mathcal{G}(n,m)$ with $\delta(G)\leq1$.\smallskip

Note that $n\geq4$ since $2m> 3n-3$.\smallskip

If $m=n+1$, then consider a graph $G_{1}$ with 4 vertices and 5 edges and a path graph $G_{2}$ with $n-3$ vertices and $n-4$ edges. Fix vertices $v_{1}\in G_{1}$ and $v_{2}\in G_{2}$. Let $G$ be the graph obtained by identifying $v_{1}$ and $v_{2}$ in a single vertex, then $G$ has $n$ vertices and $m=n+1$ edges, and $\delta(G)=\delta(G_{1})=1$. Therefore $A(n,m)\leq \delta(G) \leq 1$ and we conlude $A(n,m)=1$.\smallskip

If $m=\displaystyle{n\choose 2}$ and $G\in  \mathcal{G}(n,m)$, the $G$ is isomorphic to $K_{n}$ and $\delta(G)=1$. Therefore  $A(n,m)=1$.\smallskip

Assume now that $ n+2 \leq m< \displaystyle{n\choose 2}$. Then $m-6\geq n-4$ and we can define
$$n_{0}-1:=\max \Big  \{4\leq j \leq n-1 \mid m-\displaystyle{j\choose 2}\geq n-j \Big  \}.$$

Then $3\leq n_{0} \leq n$ and we have

$$\displaystyle{n_{0}-1\choose 2}+n-n_{0}+1\leq m < \displaystyle{n_{0}\choose 2}+n-n_{0}.  $$

Define $T:=\displaystyle{n_{0}\choose 2}+n-n_{0}-m$. Notice that 

$$1\leq T \leq \displaystyle{n_{0}\choose 2}+n-n_{0}-\displaystyle{n_{0}-1\choose 2}-n+n_{0}-1=n_{0}-2.$$

It follows from  Corollary \ref{c:comb} that there exist numbers  $t_{i}\geq2$, $i=1,\dots, s$, such that  $ \Sigma_{i} t_{i} \leq T+2 \leq n_{0}$ and $ \Sigma_{i} \displaystyle{t_{i} \choose 2} = T$. \smallskip

Choose sets of vertices $V_{1}, \dots, V_{s} \subset V(K_{n_{0}})$ with $V_{i} \cap  V_{j}= \emptyset$ if  $i\neq j$ and $\#V_{i}=t_{i}$ for $i=1,\dots, s$. Let us denote by $G_{1}$ the graph obtained from $K_{n_{0}}$ by removing the $T=\Sigma_{i} \displaystyle{t_{i} \choose 2}$ edges joining any two vertices in $V_{i}$ for every $i=1,\dots, s$. Then $G_{1}\in  \mathcal{G}(n_{0},m-n+n_{0})$ and Lemma \ref{l:c.Kn1} implies $\delta(G_{1})=\delta(K_{n_{0}}^{t_{1},\dots, t_{s}}) \leq1$. \smallskip

 Let us define $G_{2}$ as a path graph with $n-n_{0}+1$ vertices and $n-n_{0}$ edges. Fix vertices $v_{1}\in G_{1}$ and $v_{2}\in G_{2}$. Let $G$ be the graph obtained from $G_{1}$ and $G_{2}$ by identifying $v_{1}$ and $v_{2}$ in a single vertex, then $G\in \mathcal{G}(n,m)$ and $\delta(G)=\delta(G_{1})=1$. Therefore $A(n,m)\leq \delta(G) = 1$ and we conclude $A(n,m)=1$.
 
 \end{proof}
 
 The previous results have the following consequence.

 \begin{thm}\label{bound-final}
   If $m=n-1$, then $A(n,m)=0$. \\
    If $m\geq n$ and $2m\leq 3n-3$, then $A(n,m)=3/4$. \\
   If $m\geq n$ and $2m> 3n-3$, then $ A(n,m)=1$.

 \end{thm}

\section{Random graphs}
The field of random graphs was started in the late fifties and early sixties of the last
century by Erd\"os and R\'{e}nyi, see \cite{{127},{128}, {129},{130}}. 
At first, the study of random graphs was used to prove deterministic properties of graphs. For example, if we can show that a random graph has a certain property with a positive probability, then a graph must exist with this property. Lately there has been a great amount of work on the field. The practical applications of random graphs are found, for instance, in areas in which complex networks need to be modeled. See the standard references on the subject \cite{60} and \cite{jjj} for the state of the art. \smallskip

Erd\"os and R\'{e}nyi
studied  in \cite{128}    the simplest imaginable random graph, which is now named after them. 
Given  $n$ fixed vertices,  the Erd\"os-R\'{e}nyi random graph $R(n,m)$ is characterized by $m$ edges  distributed uniformly at random among all possible  $\displaystyle{n\choose 2}$  edges. However,  in order to avoid disconnected graphs, which are not geodesic metric spaces, a random tree of order $n$ is first generated and then the remaining $m-(n-1)$ edges are distributed uniformly at random over the remaining  $\displaystyle{n\choose 2}-n+1$  possible edges. Call this new model $R'(n,m)$. This modified Erd\"os-R\'{e}nyi random graph   $R'(n,m)$  has a number of desirable properties as a model of a network, see \cite{Jonck}. \smallskip

We can apply the results obtained in this work to $R'(n,m)$: \smallskip

For all $G\in R'(n,m)$ we have $A(n,m)\leq  \delta(G)  \leq B(n,m)$, and Theorems \ref{bound-final}  and  \ref{t:bound2} give the precise value for $A(n,m)$ and an upper bound of $B(n,m)$.




\begin{thebibliography}{99}

\bibitem{ABCD}
Alonso, J., Brady, T., Cooper, D., Delzant, T., Ferlini, V.,
Lustig, M., Mihalik, M., Shapiro, M. and Short, H.,
Notes on word hyperbolic groups,
in: E. Ghys, A. Haefliger, A. Verjovsky (Eds.),
Group Theory from a Geometrical Viewpoint,
World Scientific, Singapore, 1992.





\bibitem{A} Anderson, J. W., Hyperbolic geometry, 2nd edition, Springer-Verlag, London, 2005.


\bibitem{BBonk} Balogh, Z. M. and Bonk, M.,
Gromov hyperbolicity and the Kobayashi metric on strictly pseudoconvex domains,
\textit{Comment. Math. Helv.}~\textbf{75} (2000), 504-533.



\bibitem{BB} Balogh, Z. M. and Buckley, S. M.,
Geometric characterizations of Gromov hyperbolicity, {\it Invent.
Math.}~{\bf 153} (2003), 261-301.


\bibitem{Bar} Baryshnikov, Y., On the curvature of the Internet. In Workshop on Stochastic Geometry and Teletraffic,
Eindhoven, The Netherlands, April 2002.



\bibitem{Beard} Beardon, A. F., The geometry of discrete groups, Springer-Verlag, New York, 1983.





\bibitem{BRST} Bermudo, S., Rodr\'{\i}guez, J. M., Sigarreta, J. M. and Tour{\'\i}s, E.,
Hyperbolicity and complement of graphs,
{\it Appl. Math. Letters} {\bf 24} (2011), 1882-1887.


\bibitem{BRSV2} Bermudo, S.,  Rodr\'{\i}guez, J. M., Sigarreta, J. M. and Vilaire, J.-M., Gromov hyperbolic graphs, Discr. Math. 313, (2013), 1575-1585.



\bibitem{60} Bolloblas, B.,Borgs, C., Chayes, J.  and Riordan, O. Directed scale-free graphs. In Pro-ceedings of the Fourteenth Annual ACM-SIAM Symposium on Discrete Algorithms
(Baltimore, MD, 2003), pages 132 139, New York, (2003). ACM.




\bibitem{BHK} Bonk, M., Heinonen, J. and Koskela, P.,
\emph{Uniformizing Gromov hyperbolic spaces}. Ast\'erisque {\bf 270}
(2001).

\bibitem{BS} Bonk, M. and Schramm, O., Embeddings of Gromov hyperbolic spaces, {\it Geom. Funct. Anal.} {\bf 10}
(2000), 266-306.







\bibitem{BHB1} Brinkmann, G., Koolen J. and Moulton, V., On the hyperbolicity of chordal
graphs, {\it Ann. Comb.} {\bf 5} (2001), 61-69.





\bibitem{CCCR}  Carballosa, W., Casablanca, R. M., de la Cruz, A. and 
Rodr\'iguez, J. M., Gromov hyperbolicity in strong product graphs,
{\it Electr. J. Comb.} {\bf 20}(3) (2013), P2.


\bibitem{CPeRS} Carballosa, W., Pestana, D., Rodr\'{\i}guez, J. M. and Sigarreta, J. M.,
Distortion of the hyperbolicity constant of a graph,
{\it Electr. J. Comb.} {\bf 19} (2012), P67.





\bibitem{CRS2} Carballosa, W., Rodr\'{\i}guez, J. M. and Sigarreta, J. M.,
Hyperbolicity in the corona and join of graphs. Submitted.
Preprint in http://gama.uc3m.es/index.php/jomaro.html



\bibitem{CRSV} Carballosa, W., Rodr\'{\i}guez, J. M., Sigarreta, J. M. and Villeta, M.,
On the Hyperbolicity Constant of Line Graphs, 
{\it Electr. J. Comb.} {\bf 18} (2011), P210.

\bibitem{CYY} Chen, B., Yau, S.-T. and Yeh, Y.-N., Graph homotopy and Graham homotopy,
{\it Discrete Math.} {\bf 241} (2001), 153-170.

\bibitem{CDEHV}
Chepoi, V., Dragan, F. F., Estellon, B., Habib, M. and Vaxes Y.,
Notes on diameters, centers, and approximating trees of $\d$-hyperbolic geodesic spaces and graphs,
{\it Electr. Notes Discrete Math.} {\bf 31} (2008), 231-234.






\bibitem{CDP} Coornaert, M., Delzant, T. and Papadopoulos, A., G\'eometrie et th\'eorie des groupes, Lecture
Notes in Mathematics 1441, Springer-Verlag, Berlin, 1990.

\bibitem{127} Erd\"os, P. and R\'{e}nyi, A. On random graphs. I. Publ. Math. Debrecen, 6:290 297,
(1959).



\bibitem{128}  Erd\"os, P. and R\'{e}nyi, A. On the evolution of random graphs. Magyar Tud. Akad. Mat. Kutato Int.Kozl., 5:17 61, (1960).


\bibitem{130}   Erd\"os, P. and R\'{e}nyi, A. On the strength of connectedness of a random graph. Acta Math. Acad. Sci. Hungar., 12:261-267, (1961).


\bibitem{K50} Frigerio, R. and Sisto, A., Characterizing hyperbolic spaces and real trees, {\it Geom. Dedicata} {\bf 142} (2009), 139-149.

\bibitem{GH} Ghys, E. and de la Harpe, P., Sur les Groupes Hyperboliques d'apr\`es Mikhael Gromov. Progress
in Mathematics 83, Birkh\"auser Boston Inc., Boston, MA, 1990.


\bibitem{G3} Gromov, M., Asymptotic invariants of infinite groups, Geometric Group Theory, London Math.
Soc. Lecture Notes Series 182, 1993.



\bibitem{G1} Gromov, M., Hyperbolic groups, in ``Essays in group theory".
Edited by S. M. Gersten, in Math. Sci. Res. Inst. Publ. {\bf 8}. Springer, 1987, 75-265.




\bibitem{Ha} H\"ast\"o, P. A.,
Gromov hyperbolicity of the $j_G$ and $\tilde{\jmath}_G$ metrics,
{\it Proc. Amer. Math. Soc.}~{\bf 134} (2006), 1137-1142.




\bibitem{HPRT} H\"ast\"o, P. A., Portilla, A., Rodr\'{\i}guez, J. M. and Tour{\'\i}s, E.,
Gromov hyperbolic equivalence of the hyperbolic and quasihyperbolic
metrics in Denjoy domains, {\it Bull. London Math. Soc.} {\bf 42} (2010), 282-294.










\bibitem{jjj} Janson, S., Luczak,T.  and  Rucinski A., Random graphs. Wiley-Interscience Series in Discrete Mathematics and Optimization. Wiley-Interscience, New York, (2000).



\bibitem{129} Janson, S., Luczak,T.  and  Rucinski A., Random graphs. Wiley-Interscience Series in Discrete Mathematics and Optimization. Wiley-Interscience, New York, (2000).



\bibitem{Jonck} Jonckheere, E., Lohsoonthorn, P., Bonahon, F., Scaled Gromov Hyperbolic
Graphs. Wiley InterScience(www.interscience.wiley.com).
DOI 10.1002/jgt.20275 (2007).



\bibitem{K21} Jonckheere, E. A., Controle du trafic sur les reseaux a
geometrie hyperbolique--Une approche mathematique a la securite de
l'acheminement de l'information, {\it J. Europ. de Syst.
Autom.} {\bf 37} (2003), 145-159.

\bibitem{K22} Jonckheere, E. A. and Lohsoonthorn, P., Geometry of network security,
{\it American Control Conference} {\bf ACC} (2004), 111-151.


\bibitem{K24} Jonckheere, E. A., Lohsoonthorn, P. and Bonahon, F., Scaled Gromov
hyperbolic graphs, {\it J. Graph Theory} {\bf 2} (2007),
157-180.


\bibitem{K56} Koolen, J. H. and Moulton, V., Hyperbolic Bridged
Graphs, {\it Europ. J. Comb.} {\bf 23} (2002), 683-699.

\bibitem{Kr} Krantz, S. G., Complex analysis: the geometric viewpoint,
Carus Mathematical Monographs, M.A.A., Washington, 1990.





\bibitem{La} Lang, U., Extendability of large-scale Lipschitz maps, {\it Trans. Amer. Math. Soc.} {\bf 351} (1999),
3975-3988.



\bibitem{MRSV} Michel, J., Rodr\'{\i}guez, J. M., Sigarreta, J. M. and Villeta, M.,
Hyperbolicity and parameters of graphs, {\it Ars Comb.} {\bf 100} (2011), 43-63.

\bibitem{MRSV2} Michel, J., Rodr\'{\i}guez, J. M., Sigarreta, J. M. and Villeta, M.,
Gromov hyperbolicity in Cartesian product graphs,
{\it Proc. Indian Acad. Sci. Math. Sci.} {\bf 120} (2010), 1-17.


\bibitem{Mmiller}   Miller, M. and Sir\'{a}n, J.,  
\newblock Moore graphs and beyond: A survey of the degree/diameter problem, {\em Electr. J. Comb.} {\bf 20(2)} (2013), DS14. 
  

\bibitem{O} Oshika, K., Discrete groups, AMS Bookstore, 2002.

\bibitem{PeRSV} Pestana, D., Rodr\'{\i}guez, J. M., Sigarreta, J. M. and Villeta, M.,
Gromov hyperbolic cubic graphs, {\it Central Europ. J. Math.} {\bf 10(3)} (2012), 1141-1151.



\bibitem{PRSV} Portilla, A., Rodr\'{\i}guez, J. M., Sigarreta, J. M. and Vilaire, J.-M.,
Gromov hyperbolic tessellation graphs, to appear in {\it Utilitas Math.}
Preprint in http://gama.uc3m.es/index.php/jomaro.html



\bibitem{PRT1} Portilla, A., Rodr\'{\i}guez, J. M. and Tour{\'\i}s, E.,
Gromov hyperbolicity through decomposition of metric spaces II, {\it
J. Geom. Anal.}~{\bf 14} (2004), 123-149.


\bibitem{PT} Portilla, A. and Tour{\'\i}s, E.,
A characterization of Gromov hyperbolicity of surfaces with variable
negative curvature, {\it Publ. Mat.} {\bf 53} (2009), 83-110.

\bibitem{R} Rodr\'{\i}guez, J. M., Characterization of Gromov hyperbolic short graphs,
Acta Mathematica Sinica 30 (2014), 197-212.

\bibitem{RSVV} Rodr\'{\i}guez, J. M., Sigarreta, J. M., Vilaire, J.-M.
and Villeta, M., On the hyperbolicity constant in graphs,
{\it Discrete Math.} {\bf 311} (2011), 211-219.






\bibitem{Si} Sigarreta, J. M.
 \newblock 
Hyperbolicity in median graphs,
{\em Proc. Math. Sci.} {\bf 123} (2013), 455-467.



\bibitem{T} Tour{\'\i}s, E.,
Graphs and Gromov hyperbolicity of non-constant negatively curved surfaces.
{\it J. Math. Anal. Appl.} {\bf 380} (2011), 865-881.


\bibitem{Va} V\"ais\"al\"a, J., Gromov hyperbolic spaces, {\it Expo. Math.} {\bf 23} (2005), 187-231.

\bibitem{WZ} Wu, Y. and Zhang, C.,
Chordality and hyperbolicity of a graph, {\it Electr. J. Comb.} {\bf 18} (2011), P43.





\end{thebibliography}
\end{document}